\documentclass[11pt]{amsart}

\usepackage{hyperref}
\usepackage{ucs}
\usepackage[utf8]{inputenc}
\usepackage{amsmath}
\usepackage{amssymb}
\usepackage{amsthm}
\usepackage[english]{babel}
\usepackage{fontenc}
\usepackage{graphicx, tikz}
\usepackage{tikz-cd}
\usetikzlibrary{arrows, matrix}
\usepackage{amscd}

\author{Johannes Schmitt}
\title[good $G$-sets under partial geometric quotients]{A correspondence of good $G$-sets under partial geometric quotients}
\address{Departement Mathematik, ETH Z\"urich, R\"amistrasse 101, 8092 Z\"urich, Switzerland}
\email{johannes.schmitt@math.ethz.ch}
\date{\today}

\newcommand{\PP}{\ensuremath{\mathbb{P}}}

\renewcommand{\dim}{\text{dim}}

\newcommand{\sslash}{/\!\!/}

\DeclareMathOperator{\Hom}{Hom} 
\newcommand{\todo}[1]{}
\newcommand{\todoOld}[1]{}
\newcommand{\todoAlt}[1]{}
\newcommand{\todoFin}[1]{}

\newcommand{\comment}[1]{}

\newcommand{\detex}[1]{}  

\def\polhk#1{\setbox0=\hbox{#1}{\ooalign{\hidewidth
    \lower1.5ex\hbox{`}\hidewidth\crcr\unhbox0}}}

\makeatletter
\tikzset{
  edge node/.code={%
      \expandafter\def\expandafter\tikz@tonodes\expandafter{\tikz@tonodes #1}}}
\makeatother
    
\tikzset{
  subseteq/.style={
    draw=none,
    edge node={node [sloped, allow upside down, auto=false]{$\subseteq$}}},
  Subseteq/.style={
    draw=none,
    every to/.append style={
      edge node={node [sloped, allow upside down, auto=false]{$\subseteq$}}}
  }
}    

\begin{document}
\begin{abstract}
For a complex variety $\widehat X$ with an action of a reductive group $\widehat G$ and a geometric quotient $\pi: \widehat X \to X$ by a closed normal subgroup $H \subset \widehat G$, we show that open sets of $X$ admitting good quotients by $G=\widehat G / H$ correspond bijectively to open sets in $\widehat X$ with good $\widehat G$-quotients. We use this to compute GIT-chambers and their associated quotients for the diagonal action of $\text{PGL}_2$ on $(\mathbb{P}^1)^n$ in certain subcones of the $\text{PGL}_2$-effective cone via a torus action on affine space. This allows us to represent these quotients as toric varieties with fans determined by convex geometry.
\end{abstract}

 \maketitle

 \newtheoremstyle{test}
  {}
  {}
  {\it}
  {}
  {\bfseries}
  {.}
  { }
  {}
 \newtheoremstyle{test2}
  {}
  {}
  {}
  {}
  {\bfseries}
  {.}
  { }
  {}
  
 \theoremstyle{test}
\newtheorem{Def}{Definition}[section]
\newtheorem{Exa}[Def]{Example}
\newtheorem{Exe}[Def]{Exercise}
\newtheorem{Theo}[Def]{Theorem}
\newtheorem{Lem}[Def]{Lemma}
\newtheorem{Cor}[Def]{Corollary}
\newtheorem{Pro}[Def]{Proposition}
\newtheorem*{Exa*}{Example}   
\newtheorem*{Pro*}{Proposition} 
\newtheorem*{Def*}{Definition}
\newtheorem*{Cor*}{Corollary}
\newtheorem*{Lem*}{Lemma}
\newtheorem*{Theo*}{Theorem}

\theoremstyle{test2}
\newtheorem{Rmk}[Def]{Remark}
\newtheorem*{Rmk*}{Remark}   

\section{Introduction}
Let $G$ be a reductive group acting on a variety $X$, then an important question in Geometric Invariant theory is to classify the open $G$-invariant subsets $U$ of $X$ having a good quotient under the action of $G$. Define
\begin{align}
 \label{eqn:DefU} \mathcal{U}_{(X,G)} = \left\{ \parbox{0.7\linewidth}{$U \subset X$ nonempty, open $G$-invariant such that a good quotient $U \to U\sslash G$ exists (in schemes over $\mathbb{C}$)}\right\},\\
 \label{eqn:DefUpr} \mathcal{U}^{\text{pr}}_{(X,G)} = \left\{ \parbox{0.7\linewidth}{$U \subset X$ nonempty, open $G$-invariant such that a good quotient $U \to U\sslash G$ exists with $U \sslash G$ a projective variety}\right\},\\
 \label{eqn:DefUprg} \mathcal{U}^{\text{pr,g}}_{(X,G)} = \left\{ \parbox{0.7\linewidth}{$U \subset X$ nonempty, open $G$-invariant such that an affine, geometric quotient $U \to U\sslash G$ exists and such that $U \sslash G$ is a projective variety}\right\}.
\end{align}
In this note we describe a geometric situation, in which these collections of open sets for two pairs $(X,G)$, $(\widehat X, \widehat G)$ can be identified. 
\begin{Def}
 Let $G,H$ be reductive, linear algebraic groups and let $G$ act on a variety $X$. A good (resp. geometric) $H$-lift of $(X,G)$ is the data of
 \begin{itemize}
  \item a reductive algebraic group $\widehat G$ containing $H$ as a closed normal subgroup together with an identification $\widehat G / H = G$,
  \item a variety $\widehat X$ with an action of $\widehat G$,
  \item a morphism $\pi: \widehat X \to X$, which is
  \begin{itemize}
   \item $\widehat G$-equivariant with respect to the action of $\widehat G$ on $X$ induced by the action of $G$ and the morphism $\widehat G \to \widehat G / H=G$,
   \item a good (resp. geometric) quotient for the induced action of $H$ on $\widehat X$.
  \end{itemize}
 \end{itemize}
\end{Def}
Then we prove the following result.
\begin{Theo} \label{Theo:lift}
 Let $\pi: \widehat X \to X$ be a good $H$-lift of $(X,G)$ for $X$ a variety with an action of the reductive group $G$. Then the map $U \mapsto \pi^{-1}(U)$ induces an injection $\mathcal{U}_{(X,G)} \to \mathcal{U}_{(\widehat X,\widehat G)}$. Moreover, for $U \in \mathcal{U}_{(X,G)}$ the map $\pi$ induces a natural isomorphism $\pi^{-1}(U) \sslash \widehat G \cong U \sslash G$. Thus we also get an injection $ \mathcal{U}^{\text{pr}}_{(X,G)} \to \mathcal{U}^{\text{pr}}_{(\widehat X,\widehat G)}$.
 
 If $\pi$ is a geometric $H$-lift, the correspondence above is a bijection and it induces a bijection $\mathcal{U}^{\text{pr,g}}_{(X,G)} \to \mathcal{U}^{\text{pr,g}}_{(\widehat X,\widehat G)}$.
\end{Theo}
It has already been observed by various authors (see \cite[Theorem 6.1.5]{bbquotients} for an overview) that for a geometric $H$-lift $\pi: \widehat X \to X$ a $G$-invariant open subset $U \subset X$ has a good/geometric quotient by $G$ iff $\pi^{-1}(U)$ has a good/geometric quotient by $\widehat G$. Below we give a self-contained argument that also includes the case of good $H$-lifts.

The structure of the paper is as follows: as a motivation for the definition of $\mathcal{U}^{\text{pr}}_{(X,G)}$, we show in Section \ref{Sect:Linebundles} how its elements relate to $G$-linearized line bundles on $X$ for $X$ a smooth variety. This makes a connection to the classical approach of \cite{git}, obtaining (semi)stable sets from linearized line bundles. We give two special situations where we can identify the elements of $\mathcal{U}^{\text{pr,g}}_{(X,G)}$ with chambers of the $G$-effective cone in $\text{NS}^G(X)$:
\begin{itemize}
 \item for $X$ a projective homogeneous variety, $G$ reductive,
 \item for $X = \mathbb{A}^n$ and $G=T$ a torus acting linearly.
\end{itemize}
In Section \ref{Sect:Hlifts} we give the proof of Theorem \ref{Theo:lift} and describe situations where $H$-lifts appear naturally. Moreover we show that $H$-lifts where the action of $H$ on $\widehat X$ is free give an identification $\text{Pic}^G(X) \cong \text{Pic}^{\widehat G}(\widehat X)$ compatible with forming semistable sets. Finally in Section \ref{Sect:Exa} we give an explicit example, showing how to compute parts of the VGIT-decomposition of the $G$-effective cone for the componentwise action of $G=\text{PGL}_2$ on $(\mathbb{P}^1)^n$ using toric quotients.

\section*{Acknowledgements}
I want to thank Gergely B\'erczi, Brent Doran and Frances Kirwan for their advice during the preparation of this paper.

I am supported by the grant SNF-200020162928.
\section*{Conventions}
In the paper we are going to work over the complex numbers. For us, a \emph{good quotient} of the action of an algebraic group $G$ on a scheme $X$ is a morphism $p: X \to Y$ to a scheme $Y$ satisfying
\begin{enumerate}
 \item $p$ is surjective, affine and $H$-invariant,
 \item $p_*(\mathcal{O}_X^G) = \mathcal{O}_Y$, where $\mathcal{O}_X^G$ is the sheaf of $G$-invariant functions on $X$,
 \item for $Z_1, Z_2 \subset X$ closed, disjoint $G$-invariant subsets, their images $p(Z_1), p(Z_2)$ are also closed and disjoint.
\end{enumerate}
On the other hand, for $p$ to be a \emph{geometric quotient} we require the properties above, except that it is affine, but additionally we want the fibres of geometric points under $p$ to be orbits of $G$. This is the definition of \cite{git}.
\todo{Def: saturated}

\section{Motivation: Projective quotients from linearized line-bundles} \label{Sect:Linebundles}
Let $X$ be a smooth, irreducible variety with an action of a connected reductive group $G$. We want to study good quotients of open sets $U \subset X$ by $G$ which are projective varieties. 

\begin{Lem} \label{Lem:ULcorr}
 Let $X$ be a smooth, irreducible variety with an action of a connected reductive group $G$. Then the open sets $U$ in $\mathcal{U}^{\text{pr}}_{(X,G)}$ are all of the form $U=X^{ss}(L)$ for a $G$-linearized line bundle $L$ on $X$.
\end{Lem}
\begin{proof}
 By \cite[Theorem 6.1.5]{bbquotients}, all open $G$-invariant sets $U \subset X$ with a good quotient $U \sslash G$ that is a quasi-projective variety are saturated subsets of some $X^{ss}(L)$ for $L$ a $G$-linearized line bundle. Let $\pi: X^{ss}(L) \to X^{ss}(L) \sslash G$ be the corresponding good quotient. Then we have $U \sslash G \subset X^{ss}(L) \sslash G$ contained as an open subset. If $U \sslash G$ is in addition projective, this inclusion is an isomorphisms. But as $U$ is a saturated open subset, we have $U=\pi^{-1}(\pi(U))=\pi^{-1}(U \sslash G) = X^{ss}(L)$ as claimed.
\end{proof}
\begin{Rmk}
 We can generalize the setting above to $X$ being a normal variety if we work with $G$-linearized Weil divisors instead of line bundles, as described in \cite{hausenweil}. However, as our applications work with smooth $X$, we stay in the more classical setting of line bundles.
\end{Rmk}

An advantage of the condition ``$U \sslash G$ is projective'' in comparision with ``$U$ is maximal with respect to saturated inclusion in $\mathcal{U}_{(X,G)}$'' is that this can be verified intrinsically only from the action of $G$ on $U$ (without reference to the ambient variety $X$). Thus the definition of $\mathcal{U}^{\text{pr}}$ is compatible with the restriction to open $G$-invariant subsets in the following sense.
\begin{Cor} \label{Cor:chamberopen}
 Let $X$ be a variety with an action of a reductive group $G$. Let $U_0 \subset X$ be an open $G$-invariant subset. Then we have $\mathcal{U}^{\text{pr}}_{(U_0,G)} = \mathcal{U}^{\text{pr}}_{(X,G)} \cap \{U: U \subset U_0\}$ (similarly for $\mathcal{U}^{\text{pr,g}}$).
\end{Cor}

\todo{Hilbert-Mumford implies this is a subcone, doesn't it? For every point in the complement, we have the condition that it is unstable, this is a half-space on the set of line bundles.}

In the following two subsections, we are going to present situations where the dependence of $X^{ss}(L)$ of $L$ has been studied before and where the classes of $G$-linearized line bundles are partitioned into cones on which $X^{ss}(L)$ is constant. 
\subsection{Actions of reductive groups on smooth projective varieties} \label{Sect:projVGIT}
Let $X$ be an irreducible, smooth projective variety acted upon by a connected, reductive linear algebraic group $G$. In this situation, Dolgachev and Hu defined in \cite{dolgachevhu} the $G$-ample cone $C^G(X) \subset \text{NS}^G(X)_\mathbb{R}$ inside the Neron-Severi group of $G$-linearized line bundles. It is spanned by the homology classes of $G$-linearized ample line bundles $L$ such that $X^{ss}(L) \neq \emptyset$. In \cite[Theorem 3.3.2]{dolgachevhu}, they show that  if this cone has nonempty interior, it contains open chambers such that two elements $L,L' \in C^G(X)$ are in the same chamber $\sigma$ iff we have 
\[ X^{ss}(L) = X^s(L) = X^{s}(L') = X^{ss}(L')=:X^{ss}(\sigma).\]
Furthermore, as $X$ is projective, for any $L$ in a chamber as above, we have that $X^{ss}(L) \sslash G$ is projective. This shows that the set of chambers of $C^G(X)$ injects into $\mathcal{U}^{\text{pr,g}}_{(X,G)}$ by sending a chamber  $\sigma$ to  $X^{ss}(\sigma)$. Note however, that this inclusion can be strict: in \cite{bbexotic} Bia{\l}ynicki-Birula and {\'S}wi{\polhk{e}}cicka give an example of a smooth projective variety $X$ with an action of a torus $T$ together with an open set $U \subset X$, which has a projective geometric quotient $U \sslash T$ but is not of the form $X^s(L)$ for $L$ ample and $G$-linearized. For a treatment of the behaviour of $X^{ss}(L)$ when $L$ is outside the ample cone see \cite{beyondample}.

However, for certain special varieties $X$ the correspondence between chambers of $C^G(X)$ and elements of $\mathcal{U}^{\text{pr,g}}_{(X,G)}$ is bijective.
\begin{Pro} \label{Pro:Xsmoothproj}
 Let $X$ be an irreducible, smooth projective variety acted upon by a connected, reductive linear algebraic group $G$. Assume that every effective divisor is semiample (i.e. some positive power is base-point free) and that $C^G(X)$ has nonempty interior with all walls having positive codimension.
%
%
 Then the chambers of $C^G(X)$ are in bijection with $\mathcal{U}^{\text{pr,g}}_{(X,G)}$ via $\sigma \mapsto X^{ss}(\sigma)$.
\end{Pro}
\begin{proof}
 Let $U \in \mathcal{U}^{\text{pr,g}}_{(X,G)}$ then we need to show that $U$ is of the form $U=X^s(L') = X^{ss}(L')$ for some ample $G$-linearized line bundle $L'$. By Lemma \ref{Lem:ULcorr} a priori we only know that $U=X^{ss}(L)$ for some $G$-linearized (not necessarily ample) $L$. As $U \to U\sslash G$ is a geometric quotient, all orbits in $U$ are fibres of this map and hence closed, so $U=X^s(L)=X^{ss}(L)$. As $U$ is nonempty, the bundle $L$ must have at least one section, so its associated divisor is effective and thus semiample by assumption.
 
 We want to show that for $m$ sufficiently large and a suitable $L_0$ in the interior of $C^G(X)$ the line bundle $L'=L^{\otimes m} \otimes L_0$ satisfies $U \subset X^s(L')=X^{ss}(L')$. Then as in the proof of Lemma \ref{Lem:ULcorr} we see that this inclusion is already the identity. But as such $L'$ are ample and $G$-effective, this finishes the proof. In the following, we can use the Hilbert-Mumford criterion to determine the (semi)stable points of $L'$. 
 
 For this recall from \cite[Section 1.1]{dolgachevhu} the construction of the function $M^\bullet(x): \text{Pic}^G(X)_{\mathbb{R}} \to \mathbb{R}$ for $x \in X$. To define it let $\lambda : \mathbb{C}^* \to G$ be a $1$-parameter subgroup of $G$ then, as $X$ is proper, the map $\mathbb{C}^* \to X, t \mapsto \lambda(t).x$ has a limit $z$ over $t=0$. The point $z$ is fixed by $\lambda$ and for $L \in \text{Pic}^G(X)$, $\mathbb{C}^*$ acts on the fibre $L_z$ of $L$ over $z$ with weight $r=:\mu(x,\lambda)$. Let $T$ be a maximal torus in $G$ and let $\|\ \|$ be a Weyl-invariant norm on the group of $1$-parameter subgroups of $T$ tensor $\mathbb{R}$. Then for any $1$-parameter subgroup $\lambda$ of $G$ define $\| \lambda \|$ to be the norm of a suitable conjugate of $\lambda$ contained in $T$. We set
 \[M^L(x) = \sup_{\lambda \text{ 1-PSG of }G} \frac{\mu^L(x,\lambda)}{\|\lambda\|}.\]
 By \cite[Lemma 3.2.5.]{dolgachevhu} the function $M^\bullet(x)$ factors through $\text{NS}^G(X)_\mathbb{R}$ and satisfies
 \[M^{L_1 + L_2}(x) \leq M^{L_1}(x) + M^{L_2}(x), M^{m L}(x) = m M^{L}(x)\]
 for $L_1, L_2 \in \text{Pic}^G(X)$, $m>0$.
 For $L'$ ample $G$-linearized, $x$ is semistable (properly stable) with respect to $L'$ iff $M^{L'}(x) \leq 0$ ($M^{L'}(x) < 0$). 
 
 For our given $L$, we first show that $M^L(x)<0$ for all $x \in X^s(L)$. Observe that as $L$ is semiample, by \cite[Corollary 1]{semiample} there exists a $1$-parameter subgroup $\lambda$ of $G$ with $M^L(x) = \mu^L(x,\lambda)/\|\lambda\|$. Thus it suffices to show $\mu^L(x,\lambda) <0$ for all $1$-parameter subgroups $\lambda$. As $x$ is stable, there exists an invariant section $s$ of some tensor power of $L$ with $x \in X_s$ and $Gx \subset X_s$ closed. We claim that then $z=\lim_{t \to 0} \lambda(t)x$ is not contained in $X_s$. Indeed assume otherwise, then $z\in \overline{Gx} \cap X_s = Gx$, so $z=gx$ for some $g \in G$. However then the stabilizer $G_{gx}$ contains all of $\lambda(\mathbb{C}^*)$, so it is not finite. But then also the stabilizer of $x$ is not finite and we obtain a contradiction, as our assumptions imply that all stable points are properly stable\footnote{In \cite{dolgachevhu} finiteness of stabilizers was part of the definition of a stable point.}. We conclude that $s(z)=0$ and by \cite[Proposition 1]{semiample} this implies $\mu(x,\lambda)<0$. 
 
 Let $L_0 \in C^G(X)$ such that for all $0 < r \ll 1$ we have that $L + rL_0$ is contained in a chamber of $C^G(X)$. As there are only finitely many walls (\cite[Theorem 3.3.3]{dolgachevhu}), which are all of positive codimension, such $L_0$ exist. For $m \gg 0$ an integer, we have that $L'=L^{\otimes m} \otimes L_0$ is ample and for a fixed $x \in X^s(L)$ we know
 \[M^{L'}(x) \leq m M^{L}(x) + M^{L_0}(x).\]
 As $M^L(x)<0$ we can choose $m$ sufficiently big such that $M^{L'}(x)<0$ and hence $x \in X^s(L')$. The subsets
 \[Y_m = X^s(L) \setminus X^s(L^{\otimes m} \otimes L_0)\]
 form a descending chain of closed subsets of $X^s(L)$ and for all $x \in X^s(L)$ there exists $m$ with $x \notin Y_m$. By Noetherian induction we can thus choose $m_0$ such that $U=X^s(L) \subset X^s(L^{\otimes m} \otimes L_0)$ for all $m \geq m_0$. But by the choice of $L_0$ we have that $L'=L^{\otimes m} \otimes L_0$ is contained in a chamber of $C^G(X)$ for $m$ sufficiently large.
\end{proof}
The condition that every effective divisor $D$ is semiample is for instance satisfied for homogeneous projective varieties $X=G/P$. Indeed, in this case 
\[G \mapsto \text{Pic}(X), g \mapsto \mathcal{O}(g.D),\]
where $g.D$ is the translate of $D$ by $g$, is a family of line bundles over $G$. By \cite[Proposition 7]{popov}, $\text{Pic}(X)$ is discrete and hence the map above is constant and equal to $\mathcal{O}(D)$. But as the $G$-translates of $X \setminus D$ cover $X$, this shows that $\mathcal{O}(D)$ is base-point free.

\subsection{Toric quotients of affine space} \label{Sect:toric}
In this section, we explain how for linear actions $(\mathbb{C}^*)^n \curvearrowright \mathbb{C}^r$ we can compute open sets in $\mathcal{U}_{(\mathbb{C}^r,(\mathbb{C}^*)^n)}$ via elementary and algorithmically accessible operations involving fans and polyhedra. We closely follow \cite[Section 14]{toric} in notation and presentation.

Let an algebraic torus $G=(\mathbb{C}^*)^{n}$ act faithfully, linearly on the affine space $X=\mathbb{C}^r$. By a suitable change of coordinates, we may assume that $G$ acts by diagonal matrices. For $\textbf t = (t_1, \ldots, t_n) \in G$ and $\beta = (\beta^1, \ldots, \beta^n) \in \mathbb{Z}^n$ write 
\[\textbf t^\beta = t_1^{\beta^1} t_2^{\beta^2} \cdots t_n^{\beta^n}.\]
Then after coordinate change, the action of $\textbf t \in G$ on $x \in \mathbb{C}^r$ is given by
\[\textbf t.x = \text{diag}(\textbf t^{\beta_1}, \ldots, \textbf t^{\beta_r}) x\]
for integer vectors $\beta_1, \ldots, \beta_r \in \mathbb{Z}^n$. 
Note that via the identification of $\mathbb{Z}^n$ with the character group $\widehat G$ of $G$, the $\beta_i$ are simply the restrictions of the characters $\textbf{t} \mapsto t_i$ of $(\mathbb{C}^*)^r$ along the map $(\mathbb{C}^*)^n \to (\mathbb{C}^*)^r \subset \text{GL}(\mathbb{C}^r)$ specifying the action. Let 
\[\gamma: \mathbb{Z}^r \cong \widehat{(\mathbb{C}^*)^r} \to \widehat G \cong \mathbb{Z}^n\]
be this restriction map (such that $\gamma(e_i)=\beta_i$). The assumption that the action is faithful implies that $\gamma$ is surjective (\cite[Lemma 14.2.1]{toric}). Let $\delta: M \to \mathbb{Z}^r$ be the kernel of $\gamma$. Setting $N=\Hom(M,\mathbb{Z})$, the map $\delta$ is given by 
\[\delta(m)=(\langle m, \nu_1 \rangle, \ldots, \langle m, \nu_r \rangle)\]
for some $\nu_1, \ldots, \nu_r \in N$. Below we will see that the vectors $\beta_1, \ldots, \beta_r$ control the linearizations and GIT-chambers for quotients of $\mathbb{C}^r$ by $G$ and these quotients are toric varieties of fans in $N_\mathbb{R}=N \otimes_{\mathbb{Z}} \mathbb{R}$ with rays spanned by some of the vectors $\nu_1, \ldots, \nu_r$. 

For this note that, as all line bundles on $\mathbb{C}^r$ are trivial, the $G$-linearized line bundles $L=\mathbb{C}^r \times \mathbb{C} \to \mathbb{C}^r$ are specified by characters $\chi \in \widehat G$ via
\[\textbf t . (x,y) = (\textbf t.x, \chi(\textbf t) y),\text{ with }\textbf t \in G, x \in \mathbb{C}^r, y \in \mathbb{C}.\]
Denote by $(\mathbb{C}^r)^{ss}_{\chi}, (\mathbb{C}^r)^{s}_{\chi}$ the (semi)stable points with respect to these linearizations and by $\mathbb{C}^r \sslash_\chi G$ the categorical quotient of $(\mathbb{C}^r)^{ss}_{\chi}$ by $G$. 
Let $\widehat G_\mathbb{R} = \widehat G \otimes_{\mathbb{Z}} \mathbb{R}$ and let $C_\beta \subset \widehat G_\mathbb{R}$ be the cone spanned by $\beta_1 \otimes 1, \ldots, \beta_r \otimes 1$. Note that as $\gamma$ was surjective, we have $\dim\ C_\beta = \dim\ \widehat G_\mathbb{R}$. Then we have the following results:
\begin{enumerate}
 \item The set $(\mathbb{C}^r)^{ss}_{\chi}$ of semistable points is nonempty iff $\chi \otimes 1 \in C_\beta$. (\cite[Proposition 14.3.5]{toric})
 \item The set $(\mathbb{C}^r)^{s}_{\chi}$ of stable points is nonempty iff $\chi \otimes 1$ is in the interior of $C_\beta$. (\cite[Proposition 14.3.5]{toric})
 \item The quotient $\mathbb{C}^r \sslash_\chi G$ is projective for some $\chi \otimes 1 \in C_\beta$ iff all $\beta_i$ are nonzero and $C_\beta$ is strongly convex (i.e. $C_\beta \cap (- C_\beta) = \{0\}$). In this case all nonempty quotients $\mathbb{C}^r \sslash_\chi G$ are projective. (\cite[Proposition 14.3.10]{toric})
 \item We have $(\mathbb{C}^r)^{s}_{\chi} = (\mathbb{C}^r)^{ss}_{\chi}$ iff $\chi \otimes 1$ does not lie on a cone $C_{\beta'}$ generated by a subset $\beta'$ of the $\beta_i$ with $\dim\ C_{\beta'} < \dim\ C_{\beta}$. (\cite[Theorem 14.3.14]{toric})
\end{enumerate}
In fact the behaviour of $(\mathbb{C}^r)^{ss}_{\chi}$ as $\chi \otimes 1$ varies in $C_\beta$ is completely determined by the so-called secondary fan $\Sigma_{\text{GKZ}}$ (see \cite[Theorem 14.4.7]{toric}). This is a rational fan in $\widehat G_\mathbb{R}$ with support $C_\beta$ such that $(\mathbb{C}^r)^{ss}_{\chi}$ is constant for $\chi \otimes 1$ moving in the relative interior of any of the cones $\sigma \in \Sigma_{\text{GKZ}}$. 

Let $\chi \in \widehat G \cap C_\beta$ be given by $\chi = \sum_{i=1}^r a_i \beta_i$, i.e. $\chi=\gamma(\textbf{a})$. Define the polyhedron
\[P_{\textbf a} = \{m \in \mathbb{M}_\mathbb{R}: \langle m, \nu_i \rangle \geq -a_i, 1 \leq i \leq r\} \subset M_\mathbb{R}= M \otimes_{\mathbb{Z}} \mathbb{R}.\]
Let $\Sigma_\chi$ be the normal fan of $P_{\textbf a}$ (see \cite[Proposition 14.2.10]{toric}), then it is independent of the choice of ${\textbf a} \in \gamma^{-1}(\chi)$ and we have that the quotient $\mathbb{C}^r \sslash_{\chi} G$ is isomorphic to the toric variety associated to $\Sigma_\chi$ (\cite[Theorem 14.2.13]{toric}).  

Moreover, we have an explicit description of the set $(\mathbb{C}^r)^{ss}_{\chi}$ of semistable points. Set 
\[I_{\emptyset,\chi} = \{i \in \{1, \ldots, r\}: P_{\textbf a} \cap \{m: \langle m, \nu_i \rangle = -a_i\} = \emptyset \}.\]
Define the ideal
\[B(\Sigma_\chi, I_{\emptyset,\chi}) = \left(\prod_{i \notin I_{\emptyset,\chi}: \nu_i \notin \sigma} x_i: \sigma \in \Sigma_\chi \right) \cdot \left( \prod_{i \in I_{\emptyset,\chi}} x_i \right)\]
in $\mathbb{C}[x_1, \ldots, x_r]$. Then $(\mathbb{C}^r)^{ss}_{\chi} = \mathbb{C}^r \setminus V(B(\Sigma_\chi, I_{\emptyset,\chi}))$ (\cite[Corollary 14.2.22]{toric}).

In fact, the fan $\Sigma_\chi$ and the set $I_{\emptyset, \chi}$ of indices is also constant on the relative interior of the cones of $\Sigma_{\text{GKZ}}$ and these cones are uniquely indexed by this data and written as $\Gamma_{\Sigma, I_\emptyset}$.

\begin{Pro} \label{Pro:toriccorr}
 The map $\Sigma_{\text{GKZ}} \to \mathcal{U}_{(\mathbb{C}^r,G)}$ associating to a cone $\Gamma_{\Sigma, I_\emptyset}$ the set $(\mathbb{C}^r)^{ss}_{\chi}$ for any $\chi \otimes 1$ in the relative interior of $\Gamma_{\Sigma, I_\emptyset}$ is well-defined and injective. 
 
 If all vectors $\beta_i$ are nonzero and the cone $C_\beta$ is strongly convex, the map above is a bijection from $\Sigma_{\text{GKZ}}$ to $\mathcal{U}^{\text{pr}}_{(\mathbb{C}^r,G)}$ sending chambers to elements of $\mathcal{U}^{\text{pr},g}_{(\mathbb{C}^r,G)}$.
 
 Conversely, every $U \in \mathcal{U}_{(\mathbb{C}^r,G)}$ is a saturated open set of some $(\mathbb{C}^r)^{ss}_{\chi}$. 
\end{Pro}
\begin{proof}
 We have already remarked that in the relative interior of $\sigma$, the set $(\mathbb{C}^r)^{ss}_{\chi}$ is constant, so we show injectivity. Assume that $\Gamma_{\Sigma, I_\emptyset}$ and $\Gamma_{\Sigma', I_\emptyset'}$ map to the same set $U$ of semistable points. Then the vanishing ideal $I$ of $\mathbb{C}^r \setminus U$ is the radical of the ideals $B(\Sigma, I_\emptyset), B(\Sigma', I_\emptyset')$ as defined above. But as these two are ideals generated by square-free monomials, they are already radical (see for instance \cite[Corollary 1.2.5]{monomial}), so $B(\Sigma, I_\emptyset) = B(\Sigma', I_\emptyset')$. By \cite[Corollary 14.4.15]{toric} this implies $\Gamma_{\Sigma, I_\emptyset} = \Gamma_{\Sigma', I_\emptyset'}$. 
 
 The additional conditions on the $\beta_i$ guarantee that all quotients $(\mathbb{C}^r)^{ss}_\chi \sslash G$ for $\chi \otimes 1 \in C_\beta$ are projective. Assume conversely that we have $U\in \mathcal{U}^{\text{pr,g}}_{(\mathbb{C}^r,(\mathbb{C}^*)^n)}$, then by Lemma \ref{Lem:ULcorr} it is of the form $U=X^{ss}(L)$ for $L$ a $G$-linearized line bundle corresponding to the character $\chi$ of $G$. As this set is nonempty, we have $\chi \otimes 1 \in C_\beta$ and as the fan $\Sigma_{\text{GKZ}}$ has support $C_\beta$, it is contained in the relative interior of one of its cones.
 
 The last statement above is again Lemma \ref{Lem:ULcorr}. 
\end{proof}

\section{Properties of \texorpdfstring{$H$}{H}-lifts} \label{Sect:Hlifts}

%
%

We are now ready to prove Theorem \ref{Theo:lift}. For this, we need the following technical result, which we prove here in lack of a good reference.

\begin{Lem} \label{Lem:epi}
 Let $\pi: Z \to X$ be a surjective morphism of schemes with $X$ reduced. Then $\pi$ is an epimorphism, i.e. two maps $\varphi_1, \varphi_2 : X \to Y$ to some scheme $Y$ agree iff $\varphi_1 \circ \pi = \varphi_2 \circ \pi$. In particular, good quotients of reduced schemes are epimorphisms.
\end{Lem}
\begin{proof}
 For morphisms $\varphi_1, \varphi_2$ as above we have a fibred diagram
 \begin{center}
  \begin{tikzcd}
   W \arrow{r}\arrow{d} & Y\arrow{d}{\delta_Y}\\
   X \arrow{r}{(\varphi_1, \varphi_2)} & Y \times Y
  \end{tikzcd}
 \end{center}
 where $\delta_Y$ is the diagonal map of $Y$, which is a locally closed embedding. In particular $W \to X$ is also a locally closed embedding. Assume that $\varphi_1 \circ \pi = \varphi_2 \circ \pi$, then by definition the map $\pi$ factors through $W \to X$. In particular, $W \to X$ is surjective and hence a closed embedding. But as $X$ is reduced, this means that it is an isomorphism. Then the diagram above shows that $\varphi_1 = \varphi_2$. In particular, if $\pi$ is a good quotient and $Z$ is reduced, so is $X$ and thus the assumptions above are satisfied.
\end{proof}

\begin{proof}[Proof of Theorem \ref{Theo:lift}]
 We first note that as $\pi$ is surjective, the $G$-invariant subsets $U$ of $X$ inject via $\pi^{-1}$ into the $\widehat G$-invariant subsets $\widehat U$ of $\widehat X$. If $\pi$ is a geometric quotient this is a bijection with  inverse map given by $\widehat U \mapsto \pi(\widehat U)$. This is well-defined because $\pi$ sends open $H$-invariant sets to open sets and it is an inverse to $\pi^{-1}$ as the fibres of $\pi$ are orbits. 
 
 Before we continue, recall the following fact, which is Lemma 5.1 in \cite{ramanathan}. Let a reductive algebraic group $G'$ act on schemes $Y,Z$. If $Y \to Z$ is an affine, $G'$-equivariant morphism and $Z \to Z \sslash G'$ is a good quotient, then $Y$ also has a good quotient $Y \to Y \sslash G'$ and the induced morphism $Y \sslash G' \to Z \sslash G'$ is affine. 
 
 First assume that $U \in \mathcal{U}_{(X,G)}$, so we have a good quotient $U \to U \sslash G = U \sslash \widehat G$. The map $\pi|_{\pi^{-1}(U)}$ is $\widehat G$-equivariant and affine. Then by the result above, $\pi^{-1}(U)$ has a good quotient $\pi^{-1}(U) \sslash \widehat G$, which maps to $U \sslash G$ via a map $\psi$.
 We want to show $\psi$ is an isomorphism, so we construct an inverse $\varphi$. The $H$-invariant map $\pi^{-1}(U) \to \pi^{-1}(U) \sslash \widehat G$ factors uniquely through a map $\varphi': U \to \pi^{-1}(U) \sslash \widehat G$ as $\pi$ is a categorical $H$-quotient and $\varphi'$ is $G$-invariant. But as $U \to U \sslash G$ is a quotient for the $G$-action on $U$, the map $\varphi'$ factors uniquely through some map $\varphi: U \sslash G \to \pi^{-1}(U) \sslash \widehat G$.  We can write the following commutative diagram
 \begin{center}
  \begin{tikzcd}
   \pi^{-1}(U) \arrow{r}{\pi} \arrow{d} & U \arrow{d} \arrow{dl}{\varphi'} & \pi^{-1}(U)\arrow{l} \arrow{r}\arrow{d} & U\arrow{d}\\
   \pi^{-1}(U) \sslash \widehat G & U \sslash G \arrow{l}{\varphi} & \pi^{-1}(U) \sslash \widehat G \arrow{l}{\psi} & U \sslash G \arrow{l}{\varphi}
  \end{tikzcd}
 \end{center}
 Via diagram chase and using that good quotients of reduced schemes are epimorphisms (Lemma \ref{Lem:epi}), we conclude that $\varphi \circ \psi$ and $\psi \circ \varphi$ are both the identity on their domains.
 
 If $U \to U \sslash G$ and $\pi$ are geometric quotients, then the preimage of some geometric point $p \in U \sslash G$ in $U$ is a $G$-orbit and thus its preimage in $\pi^{-1}(U)$ is a $\widehat G$-orbit, hence $\pi^{-1}(U) \to U \sslash G$ is a geometric quotient.
 
 Now assume that $\pi^{-1}(U)$ has a good quotient map $\pi^{-1}(U) \to \pi^{-1}(U) \sslash \widehat G$. For the trivial $H$-action on the latter space, this is a $H$-equivariant affine map and clearly the identity on $\pi^{-1}(U) \sslash \widehat G$ is a good quotient for the trivial $H$-action. Thus by the result from \cite{ramanathan}, the $H$-action on $\pi^{-1}(U)$ has a good quotient (which is isomorphic to $U$, as $\pi$ is a good $H$-quotient) and the map $\psi: U \to \pi^{-1}(U) \sslash \widehat G$ is affine. To show that it is a good quotient, we consider the diagram
 \begin{center}
  \begin{tikzcd}
   \pi^{-1}(U) \arrow{r}{\pi} \arrow{d} & U\arrow{dl}{\psi}\\
   \pi^{-1}(U) \sslash \widehat G
  \end{tikzcd}
 \end{center}
 and use that $\pi$ is an epimorphism to show that $\psi$ is surjective, $G$-invariant and sends disjoint closed $G$-invariant sets to disjoint closed sets. Given a $G$-invariant local function $f$ on $U$, the function $f \circ \pi$ is $\widehat G$-invariant, so it factors uniquely through some function $g$ on $\pi^{-1}(U) \sslash \widehat G$. Again using that $\pi$ is an epimorphism, we see $f=g \circ \psi$, so indeed $f$ factors through $\pi^{-1}(U) \sslash \widehat G$. Thus $\psi$ is a good $G$-quotient.
 
 If $\pi^{-1}(U) \to \pi^{-1}(U) \sslash \widehat G$ is a geometric quotient, its geometric fibres are orbits of $\widehat G$, so the fibres in $U$ are $G$-orbits and thus $\psi$ is a geometric quotient.
\end{proof}
Instead of looking at the correspondence of open sets admitting a good quotient induced by $H$-lifts, we can also directly consider the behaviour of equivariant Picard groups and the corresponding (semi)stable sets. Here we have the following result.
\begin{Pro} \label{Pro:Piccorr}
 Let $\pi: \widehat X \to X$ be a good $H$-lift of $(X,G)$ for $X$ a variety with an action of the reductive group $G$. Then pullback by $\pi$ induces an map $\pi^* : \text{Pic}^G(X) \to \text{Pic}^{\widehat G}(\widehat X)$ and we have
 \[\widehat X^{ss}(\pi^* L) = \pi^{-1}(X^{ss}(L))\]
 for $L \in \text{Pic}^G(X)$. If $\pi$ is a geometric $H$-lift and $H$ acts freely on $\widehat X$, the map $\pi^*$ is an isomorphism.
\end{Pro}
\begin{proof}
 Via the map $\widehat G \to G=\widehat G / H$ we have a natural map $\text{Pic}^G(X) \to \text{Pic}^{\widehat G}(X)$ by extending a $G$-action on a line bundle $L$ to a $\widehat G$-action. For the $\widehat G$-equivariant morphism $\pi$ we then have a natural pullback map $\text{Pic}^{\widehat G}(X) \to \text{Pic}^{\widehat G}(\widehat X)$ and the map $\pi^*$ above is the composition of these two homomorphisms. 
 
 Fix $L$ (the total space of) a $G$-linearized line bundle on $X$ and let $p: L \to X$ be the corresponding $G$-equivariant morphism. Then we have a cartesian diagram
 \begin{equation*}
 \begin{tikzcd}
  \pi^*(L) \arrow{r} \arrow{d}{\widehat p} & L\arrow{d}{p}\\
  \widehat X \arrow{r}{\pi} & X
  \end{tikzcd}
 \end{equation*}
 where all maps are $\widehat G$-equivariant. Now $G$-invariant global sections of $L$ are $G$-equivariant sections of $p$ and those correspond bijectively to $\widehat G$-equivariant sections of $\widehat p$, i.e. global sections of $\pi^*(L)$. Here we use that $\pi$ is a categorical $H$-quotient. Thus $\pi^*$ induces a natural isomorphism $\Gamma(L)^G \cong \Gamma(\pi^*(L))^{\widehat G}$. Of course this argument also works after replacing $L$ by $L^{\otimes k}$ for $k \geq 1$.
 
 Now let $x \in X^{ss}(L)$, then there exists a $G$-invariant section $s$ of some $L^{\otimes k}$ with $x \in X_s = \{x':s(x') \neq 0\}$ and $X_s$ is affine. But then $\pi^*s$ is a $\widehat G$-invariant section of $\pi^* L^{\otimes k}$ and $\widehat X_{\pi^* s} = \pi^{-1}(X_s)$ is affine as $\pi$ is an affine morphism. Hence all elements of $\pi^{-1}(x)$ are $\pi^*(L)$-semistable.
 
 Conversely for $\widehat x \in \widehat X^{ss}(\pi^*L)$ there exists a $\widehat G$-invariant section $\widehat s$ of some $\pi^* L^{\otimes k}$ with $\widehat x \in \widehat X_{\widehat s}$ and $\widehat X_{\widehat s}$ affine. By the argument above, $\widehat s = \pi^* s$ for some $G$-invariant section $s$ of $L^{\otimes k}$ and we only need to show $X_s$ affine. But clearly $\widehat X_{\widehat s} \to X_s$ is a categorical quotient of the affine variety $\widehat X_{\widehat s}$ by $H$ and thus $X_s$ is affine by \cite[Theorem 1.1]{git}. Hence $\pi(\widehat x)$ is $L$-semistable.
 
 If the action of $H$ is free on $\widehat X$, by \cite[Proposition 0.9]{git} the map $\pi$ is a fppf-locally trivial $H$-torsor. The fact that $\pi^*$ is an isomorphism $\text{Pic}^G(X) \to \text{Pic}^{\widehat G}(\widehat X)$ then follows from descent along torsors. A concise way to put the proof, using the language of stacks, is the following: the fact that $\pi$ is a $H$-torsor implies that there is a canonical isomorphism $X \cong [\widehat X/H]$. Taking the quotient stack under the actions of $G=\widehat G / H$ on both sides we have
 \[[X/G] \cong [[\widehat X/H]/(\widehat G/H)] \cong [\widehat X / \widehat G],\]
 where in the last isomorphism we use \cite[Remark 2.4]{romagny2005}. Taking Picard groups on both sides we see 
 \[\text{Pic}^G(X) = \text{Pic}([X/G]) \cong \text{Pic}([\widehat X / \widehat G]) = \text{Pic}^{\widehat G}(\widehat X)\]
 and this isomorphism is exactly given by pullback via $\pi$.
\end{proof}
In the example presented in Section \ref{Sect:Exa}, all $H$-lifts that are used will come from a free $H$-action on $\widehat X$, so we have isomorphisms of Picard groups as above.

\section{Applications}
In this section we will see several situations, where $H$-lifts naturally appear and thus allow us to conclude results about the chamber-decompositions of $G$-effective cones.
\subsection{Partial quotients}
One possibility to construct $H$-lifts is basically a reformulation of the definition.
\begin{Pro} \label{Pro:partquot}
 Let $\widehat X$ be a variety acted upon by a reductive group $\widehat G$ and assume a closed, normal subgroup $H \subset \widehat G$ acts on $\widehat X$ with a good (resp. geometric) quotient $\pi: \widehat X \to X$, where $X$ is a variety. Then $X$ carries an induced action of $\widehat G/H$ making $(\widehat X, \widehat G)$ a good (resp. geometric) $H$-lift of $(X,\widehat G/H)$. 
\end{Pro}
Combined with Corollary \ref{Cor:chamberopen}, this tells us the following: assume we are given a $\widehat G$-action on $\widehat X$ and a closed normal subgroup $H$ of $\widehat G$ acting on the open, $\widehat G$-invariant set $U_0 \subset \widehat X$ with geometric quotient $U_0 \sslash H$. Then the 
open sets $U \in \mathcal{U}^{\text{pr,g}}_{(\widehat X, \widehat G)}$ contained in $U_0$ are in bijection with $\mathcal{U}^{\text{pr,g}}_{(U_0 \sslash H,\widehat G/H)}$, which is a problem on a smaller-dimensional variety. 

\subsection{Morphisms to homogeneous spaces}
\begin{Pro} \label{Pro:homspace}
 Let a reductive group $G$ act on an irreducible variety $X$ and assume we are given a $G$-equivariant morphism $\varphi: X \to Z$ to a homogeneous $G$-space $Z$ (i.e. the action of $G$ on $Z$ is transitive). Let $z_0 \in Z$ be a closed point and let $H=G_{z_0}$ be its stabilizer in $G$, which we assume to be reductive. Consider the variety 
 \[\widehat X = \{(g,x) : \varphi(g x) = z_0\} \subset G \times X\]
 with the action of $G \times H$ given by
 \[(g',h).(g,x) = (h g(g')^{-1},g'x).\]
 Then the projection 
 \[\pi_X : \widehat X \to X, (g,x) \mapsto x\]
 makes $(\widehat X, G \times H)$ a geometric $H$-lift for $(X,G)$. On the other hand, for $Y=\pi^{-1}(z_0) \subset X$ with the induced action of $H=G_{z_0}$, the map 
 \[\pi_Y : \widehat X \to Y, (g,x) \mapsto gx\]
 makes $(\widehat X, G \times H)$ a geometric $G$-lift of $(Y,H)$.
\end{Pro}
\begin{proof}
 By definition of $Y$ and $\widehat X$, we have cartesian diagrams
 \begin{equation}\label{eqn:DefhatXY}
 \begin{tikzcd}
  \widehat X \arrow{r}{\pi_Y} \arrow{d} & Y \arrow{r} \arrow{d} &\{z_0\} \arrow{d}\\
  G \times X \arrow{r}{\sigma} & X \arrow{r}{\varphi} & Z
  \end{tikzcd}
 \end{equation}
 where $\sigma$ is the action map of $G$ on $X$. As we are in characteristic zero and as $G \times X$ and $X$ are irreducible, the fibres of the generic point $\eta_Z$ of $Z$ under $\varphi$ and $\varphi \circ \sigma$ are geometrically reduced. Hence by \cite[Theorem 9.7.7]{egaIV3}, the set of closed points in $Z$ whose fibre under $\varphi$ and $\varphi \circ \sigma$ is geometrically reduced is open and nonempty. But because the $G$-action on $Z$ is transitive, all these fibres are isomorphic to $Y$ and $\widehat X$, respectively. Thus these are varieties over $\mathbb{C}$. 
 
 From the formula for the action of $G \times H$ on $\widehat X$, it is clear that the maps $\pi_X, \pi_Y$ are $G \times H$-equivariant for the induced actions of $G=G\times H / H$ on $X$ and $H=G \times H / G$ on $Y$.
 
 For the map $\pi_X$, observe that we can obtain it using a different cartesian diagram, namely
 \[\begin{tikzcd}
    \widehat X \arrow{r} \arrow{d}{\pi_X} & G \arrow{d}{\psi}\\
    X \arrow{r}{\varphi} & Z
   \end{tikzcd}
\]
 where $\psi(g) = g^{-1} z_0$. Clearly $\psi$ is a fpqc-locally trivial $H$-torsor representing $Z$ as the quotient $G/H$. But then its base change $\pi_X$ via $\varphi$ is still a fpqc-locally trivial $H$-torsor and thus a geometric quotient.
 
 On the other hand, for $\pi_Y$ we see from the diagram (\ref{eqn:DefhatXY}) that it is a base change of the map $\sigma$, which clearly is a (trivial) $G$-torsor (using the automorphism $(g,x) \mapsto (g,g^{-1}x)$ of $G \times X$). Thus it is a (trivial) $G$-torsor itself and hence a geometric quotient.
\end{proof}
Using Theorem \ref{Theo:lift} we see that given a $G$-action on a variety $X$ and a subset $Y$ of $X$ obtained as the fibre of an equivariant map to a $G$-homogeneous space, we have a bijection between $\mathcal{U}_{(X,G)}$ and $\mathcal{U}_{(Y,G_Y)}$, where $G_Y$ is the subgroup of $G$ leaving $Y$ stable. 

Note that in \cite[Definition 15.1]{bbquotients} a subvariety $Y \subset X$ such that for all $y \in Y$ we have $H=\{g \in G: gy \in Y\}$ is called a strong $H$-section of $X$. If $X$ is normal, \cite[Lemma 15.2]{bbquotients} says that for such a $Y$ the morphism $G \times_H Y \to X$ given by $[(g,y)] \mapsto gy$ is a $G$-isomorphism. Using this isomorphism, we have a $G$-equivariant projection $X \cong G \times_H Y \to G / H$ with fibre $Y$ over $[e] \in G /H$, placing us in the situation of Proposition \ref{Pro:homspace}. Again it has been noted before that $Y$ has a good/geometric $H$-quotient iff $X$ has a good/geometric $G$-quotient (\cite[Corollary 15.3]{bbquotients}).

\section{Example} \label{Sect:Exa}
To illustrate how the techniques above can be used in practice, consider the diagonal action of $G=\text{PGL}_2$ on $X=(\mathbb{P}^1)^n$. We demonstrate how some of the chambers of the $G$-effective cone can be related to chambers of the cone $C_\beta$ for a linear action of $(\mathbb{C}^*)^{n-1}$ on $\mathbb{C}^{2n-4}$. Here we can compute the chamber decomposition as well as the resulting quotient varieties using the toric methods we recalled in Section \ref{Sect:toric}. 

The action of $G$ on $X$ has been intensely studied in the past(\cite{git}, \cite{polito}, \cite{ballquot}, \cite{weighted}). The line bundle $\mathcal{O}(a_1, \ldots, a_n)$ on $X$ carries a (unique) $G$-linearization iff the sum of the $a_i$ is even, so
\[\text{Pic}^G(X) \cong \left\{(a_1, \ldots, a_n) \in  \mathbb{Z}^n: \sum_{i=1}^n a_i \equiv 0 \text{ mod } 2\right\} \subset \mathbb{Z}^n\]
and the effective cone is given by
\[(\mathbb{R}_{\geq 0})^n \subset \mathbb{R}^n = \text{Pic}^G(X)_{\mathbb{R}}.\]
We can analyze (semi)stability with respect to a given polarization using the Hilbert-Mumford numerical criterion.
For $a=(a_1, \ldots, a_n) \in (\mathbb{Z}_{>0})^n$ with $|a|=\sum_{i=1}^n a_i$ even, a point $p=(p_1, \ldots, p_n) \in X$ is semistable with respect to $\mathcal{O}(a_1, \ldots, a_n)$ iff for all $p \in \PP^1$ we have $\sum_{i: p_i=p} a_i \leq |a|/2$. The point $p$ is stable iff all inequalities above are strict. From this we see that the $G$-effective ample cone is given by
\[C^G(X) = \left\{(a_1, \ldots, a_n) \in \mathbb{R}^n: 0 < a_i \leq \sum_{j \neq i} a_j\right\} \subset \mathbb{R}^n.\]
The criterion above also gives an explicit identification of the VGIT-chamber structure of $C^G(X)$. For $S \subset \{1, \ldots, n\}$ consider the half-space
\[H_S = \left\{(a_1, \ldots, a_n) \in \mathbb{R}^n: \sum_{i \in S} a_i \geq  \sum_{i \in \{1,\ldots, n\} \setminus S} a_i\right\} \subset \mathbb{R}^n.\]
Then the hyperplanes corresponding to the half-spaces above divide $C^G(X)$ into connected components, which are exactly the chambers of the VGIT-decomposition as in Section \ref{Sect:projVGIT}.

Though it is easy to determine the various chambers, it is more difficult to compute the quotients associated to them. In \cite{ballquot}, some of the quotients are computed for $n=5,6,7,8$. Using the techniques from the previous sections, we are able to compute these quotients for chambers contained in certain subcones of $C^G(X)$. 

For the notation below, recall from Section \ref{Sect:toric} that given an action of a torus $T$ on $\mathbb{C}^k$, the set of linearizations of the action is given by the character group $\widehat T$. Inside $\widehat T \otimes \mathbb{R}$ we have a fan $\Sigma_{\text{GKZ}}$ such that the set of $\chi$-semistable points in $\mathbb{C}^k$ is constant as the linearization $\chi$ varies in the relative interior of the cones of $\Sigma_{\text{GKZ}}$, which are denoted by $\Gamma_{\Sigma,I_\emptyset}$. 
\begin{Theo} \label{Theo:Sl2Pn}
 Let $S \subset \{1, \ldots, n\}$ with $|S|=2$, then the chambers $\sigma$ of $C^G(X)$ contained in $H_S$ are in bijective correspondence to the chambers $\Gamma_{\Sigma, I_\emptyset} \in \Sigma_{\text{GKZ}}$ for the action of $T=(\mathbb{C}^*)^{n-1}$ on $\mathbb{C}^{2n-4}$ given by
 \begin{align}
   &(t_1, \ldots, t_{n-2},s).(x_1, y_1, x_2, y_2, \ldots, x_{n-2}, y_{n-2}) \label{eqn:Taction}\\
   =&(t_1 x_1, s t_1 y_1, t_2 x_2, s t_2 y_2, \ldots, t_{n-2} x_{n-2},s t_{n-2} y_{n-2}). \nonumber
 \end{align}
 Under this correspondence, the quotient variety associated to $\sigma$ is the toric variety associated to the fan $\Sigma$. 
\end{Theo}
\begin{proof}
 Assume for simplicity of notation $S=\{1,2\}$ below. As $X$ is smooth and as every effective divisor is semiample on $X$, by Proposition \ref{Pro:Xsmoothproj} the chambers of $C^G(X)$ are in bijection with the open sets in $\mathcal{U}^{\text{pr,g}}_{(X,G)}$ by sending $\sigma$ to $X^{ss}_\sigma$. Now for any $a \in C^G(X) \cap \mathrm{int}(H_S)$ and $p \in X$ semistable with respect to $a$ we know $p_1 \neq p_2$ by the Hilbert-Mumford criterion. Thus under the correspondence above, the chambers contained in $H_S$ correspond to the subset 
 \[\mathcal{U}^{\text{pr,g}}_{((\PP^1)^{n} \setminus \Delta_{12},G)} \subset \mathcal{U}^{\text{pr,g}}_{(X,G)},\]
 where $\Delta_{12}=\{(p_1, \ldots, p_n): p_1 = p_2\}$. So the projection $\varphi: (\PP^1)^{n} \setminus \Delta_{12} \to \PP^1 \times \PP^1 \setminus \Delta=Z$ to the first two factors is a $G$-equivariant morphism to the homogeneous $G$-space $Z$. For $z_0=([0:1],[1:0]) \in Z$, the stabilizer $G_{z_0}$ is exactly the diagonal torus 
 \[H=\mathbb{C}^*=\left\{\left[ \begin{pmatrix} 1 & 0\\0 & a \end{pmatrix} \right] : a \in \mathbb{C}^* \right\} \subset \text{PGL}_2.\]
 By Proposition \ref{Pro:homspace} we obtain a variety $\widehat X$ with an action of $G \times H$ which is a geometric $H$-lift for $(X,G)$ and a geometric $G$-lift for $(\varphi^{-1}(z_0), H) = ((\PP^1)^{n-2}, \mathbb{C}^*)$. Here the action of $\mathbb{C}^*$ on $(\mathbb{P}^1)^{n-2}$ is given by 
 \[a.([x_1,y_1], \ldots, [x_{n-2},y_{n-2}]) = ([x_1,a y_1], \ldots, [x_{n-2},a y_{n-2}]).\]
 By Theorem \ref{Theo:lift} the geometric $H$ and $G$-lifts above give a natural bijection 
 \[\mathcal{U}^{\text{pr,g}}_{((\PP^1)^{n} \setminus \Delta_{12},G)} = \mathcal{U}^{\text{pr,g}}_{((\PP^1)^{n-2},\mathbb{C}^*)}.\]
 We now approach the pair $((\PP^1)^{n-2}, \mathbb{C}^*)$ from a different angle. Of course the space $(\PP^1)^{n-2}$ is a geometric quotient of $(\mathbb{C}^2 \setminus \{0\})^{n-2}$ by $(\mathbb{C}^*)^{n-2}$ via the action
 \begin{align*}
   &(t_1, \ldots, t_{n-2}).(x_1,y_1,  \ldots, x_{n-2}, y_{n-2})\\
   =&(t_1 x_1, t_1 y_1, t_2 x_2, t_2 y_2, \ldots, t_{n-2} x_{n-2}, t_{n-2} y_{n-2}).
 \end{align*}
 The action of $\mathbb{C}^*$ on $(\PP^1)^{n-2}$ lifts to a linear action on the prequotient $\mathbb{C}^{2n-4}$, which commutes with the action of $(\mathbb{C}^*)^{n-2}$ above and together they determine the action of $(\mathbb{C}^*)^{n-1}$ given in (\ref{eqn:Taction}).
 By Proposition \ref{Pro:toriccorr} the chambers of the secondary fan $\Sigma_{\text{GKZ}}$ for this toric action correspond to elements of $\mathcal{U}^{\text{pr,g}}_{(\mathbb{C}^{2n-4}, (\mathbb{C}^*)^{n-1})}$ by sending a chamber $\Gamma_{\Sigma, I_{\emptyset}}$ to $(\mathbb{C}^{2n-4})^{ss}_{\chi}=(\mathbb{C}^{2n-4})^{s}_{\chi}$ for any $\chi$ contained in this chamber.
 
 However, for the action above no point $(x_1, y_1, \ldots, x_n, y_n)$ with $x_i = y_i=0$ for some $i$ can be stable (with respect to any character) as it has nonfinite stabilizer. Thus we have
 \[\mathcal{U}^{\text{pr,g}}_{(\mathbb{C}^{2n-4}, (\mathbb{C}^*)^{n-1})} = \mathcal{U}^{\text{pr,g}}_{((\mathbb{C}^2 \setminus \{0\})^{n-2}, (\mathbb{C}^*)^{n-1})}.\]
 Using Proposition \ref{Pro:partquot} the space $((\mathbb{C}^2 \setminus \{0\})^{n-2}, (\mathbb{C}^*)^{n-1})$ is a $(\mathbb{C}^*)^{n-2}$-lift of $((\PP^1)^{n-2}, \mathbb{C}^*)$ as above, so we can identify
 \[\mathcal{U}^{\text{pr,g}}_{((\PP^1)^{n-2},\mathbb{C}^*)} = \mathcal{U}^{\text{pr,g}}_{((\mathbb{C}^2 \setminus \{0\})^{n-2}, (\mathbb{C}^*)^{n-1})}.\]
 Combining the correspondences above (see also the diagram in Remark \ref{Rmk:overview}) we have proved the claim.
\end{proof}
\begin{Rmk} \label{Rmk:overview}
 In the situation of Theorem \ref{Theo:Sl2Pn} we can not only relate chambers of the $G$-effective cones for $((\PP^1)^n,\text{PGL}_2)$ and $(\mathbb{C}^{2n-4},(\mathbb{C}^*)^{n-1})$ abstractly but we can actually find a linear map between the equivariant Picard groups inducing this correspondence. Recall from Proposition \ref{Pro:Piccorr} that for $\pi: \widehat X \to X$ a geometric $H$-lift with respect to a free $H$-action, the pullback by $\pi$ induces an isomorphism $\pi^*: \text{Pic}^G(X) \to \text{Pic}^{\widehat G}(\widehat X)$ with $\pi^{-1}(X^{ss}(L))=\widehat X^{ss}(\pi^* L)$ for $L \in \text{Pic}^G(X)$. We illustrate again the course of the proof of Theorem \ref{Theo:Sl2Pn}.
 \[
  \begin{tikzcd}[column sep=small]
   \text{PGL}_2 \curvearrowright (\PP^1)^n& & (\mathbb{C}^*)^{n-1} \curvearrowright \mathbb{C}^{2n-4}\\
   \text{PGL}_2 \curvearrowright (\PP^1)^{n} \setminus \Delta_{12} \arrow[Subseteq]{u}{} \arrow{rd}& & (\mathbb{C}^*)^{n-1} \curvearrowright (\mathbb{C}^2 \setminus \{0\})^{n-2}\arrow[Subseteq]{u}{} \arrow{ld}\\
   & \mathbb{C}^* \curvearrowright (\mathbb{P}^1)^{n-2}
  \end{tikzcd}
 \]
 Both arrows at the bottom are (compositions of) geometric $H$-lifts for free $H$-actions, so they induce isomorphisms of equivariant Picard groups compatible with forming semistable sets. We have to see how the two inclusions at the top behave in this respect.
 
 The inclusion $(\mathbb{C}^2 \setminus \{0\})^{n-2} \subset \mathbb{C}^{2n-4}$ has complement of codimension $2$, so it induces isomorphisms of (equivariant) Picard groups and (invariant) sections of line bundles. Also the complement of the inclusion above consists of points with nonfinite stabilizers. So for every linearization on $\mathbb{C}^{2n-4}$ such that stable and semistable points agree, these sets are anyway contained in  $(\mathbb{C}^2 \setminus \{0\})^{n-2}$. Thus on the interior of the chambers in $\text{Pic}^{(\mathbb{C}^*)^{n-1}}(\mathbb{C}^{2n-4})_\mathbb{Q}$ the isomorphism above respects the formation of (semi)stable points. Note that this is not true for all linearizations: for the trivial linearization all of $\mathbb{C}^{2n-4}$ is semistable, but on $(\mathbb{C}^2 \setminus \{0\})^{n-2}$ the trivial linearization has no semistable points (as this variety is not affine).
 
 For the other inclusion $i: (\PP^1)^{n} \setminus \Delta_{12} \hookrightarrow (\PP^1)^n$ we have
 \[\text{Pic}^G((\PP^1)^n \setminus \Delta_{12})_\mathbb{Q} = \text{Pic}^G((\PP^1)^n)_\mathbb{Q} / \mathbb{Q} \mathcal{O}(1,1,0, \ldots, 0)\]
 and $i^*$ is the corresponding quotient map. For any $G$-linearized line bundle $L'$ on $(\PP^1)^n \setminus \Delta_{12}$, which is the restriction of a bundle $L$ on $(\PP^1)^n$, any invariant section $s'$ of $(L')^{\otimes k}$ extends to a section $s$ of $(L \otimes\mathcal{O}(m,m,0, \ldots, 0))^{\otimes k}$ vanishing on $\Delta_{12}$ for $m \gg 0$ (take $m$ greater than the order of the rational section $s$ of $L^{\otimes k}$ along $\Delta_{12}$). Conversely, for $L=\mathcal{O}(a_1, a_2, \ldots, a_n)$ on $(\PP^1)^n$ with $a_1 + a_2 > a_3 + \ldots + a_n$ we consider again the Hilbert-Mumford criterion from above. For $S \subset \{1, \ldots, n\}$ and $\Sigma_S(\textbf a) = \Sigma_{s \in S} a_s - \Sigma_{s \notin S} a_s$ we have
 \begin{itemize}
  \item $\Sigma_{S}(\textbf{a})>0$ for $1,2 \in S$,
  \item $\Sigma_{S}(\textbf{a})<0$ for $1, 2 \notin S$,
  \item $\Sigma_{S}(\textbf{a}) = \Sigma_{S}(\textbf{a}+(m,m,0, \ldots, 0))$ for $1 \in S, 2 \notin S$ or $1 \notin S, 2 \in S$.
 \end{itemize}
 So we see that twisting $L$ by $\mathcal{O}(m,m,0, \ldots, 0)$ does not change the set of semistable points. In fact this shows that all cones of the VGIT-fan in $\text{Pic}^G((\PP^1)^n)_{\mathbb{R}}$ with relative interior strictly inside the interior of the half-space $H_{\{1,2\}}=\{a:\Sigma_{\{1,2\}}\geq 0\}$ have the cone generated by $\mathcal{O}(1,1, 0, \ldots, 0)$ in their closure and thus as a face. 
 Moreover, the set of semistable points of $L$ is contained in $(\PP^1)^n \setminus \Delta_{12}$. All this shows that for any $L \in \text{Pic}^G((\PP^1)^n)$ and $L' = i^* L$ its restriction to $(\PP^1)^n \setminus \Delta_{12}$, we have
 \[((\PP^1)^n \setminus \Delta_{12})^{ss}_{L'} = ((\PP^1)^n)^{ss}_{L\otimes\mathcal{O}(m,m,0, \ldots, 0)} \]
 for $m \gg 0$. To conclude, inside $\text{Pic}^G((\PP^1)^n)_{\mathbb{R}}$ we have the subfan of the VGIT-fan contained in $H_{\{1,2\}}$. Via the map $i^*$ it maps to its quotient fan by the ray $\text{Cone}(\mathcal{O}(1,1, 0, \ldots, 0))$. Moreover, on the relative interior of the cones in the quotient fan, the set of semistable points is constant and equal to the semistable points on the cone in the preimage containing $\mathcal{O}(1,1, 0, \ldots, 0)$.
\end{Rmk}

\begin{Rmk}
The linear action of $T=(\mathbb{C}^*)^{n-1}$ on $\mathbb{C}^{2n-4}$ that arises above has been studied in the Master thesis \cite{toricblowup} of the author. It arises as the canonical representation of the toric variety $\text{Bl}_{n-2} \mathbb{P}^{n-3}$ as a torus quotient of affine space with respect to a symmetric linearization (i.e. the character $(1, \ldots, 1)$ of $T$). In the thesis a family of chambers of the secondary fan together with their quotients is explicitly identified. The quotients occurring in this family are iterated projective $\PP^1$-bundles over some $\text{Bl}_{k} \mathbb{P}^{n'-3}$ $(k \leq n'-2)$. 
\end{Rmk}

\bibliographystyle{alpha}
\bibliography{Biblio}

\end{document}